\documentclass[a4paper,12pt]{article}
\usepackage{ayv_paperstyle}
\begin{document}
\title{Deterministic limit of mean field games associated with nonlinear Markov processes}
\author{Yurii Averboukh\footnote{Krasovskii Intitute of Mathematics and Mechanics UrB RAS and Ural Federal University, 16, S.Kovalevskaya, Yekaterinburg, Russia. e-mail: ayv@imm.uran.ru}}
\date{}
\maketitle
\begin{abstract}The paper is concerned with the deterministic limit of  mean field games with a nonlocal coupling. It is assumed that the dynamics of  mean field games are given by  nonlinear Markov processes.   This type of games includes stochastic mean field games as well as mean field games with finite state space.  We consider the limiting deterministic mean field game within the framework of minimax approach.  The concept of minimax solutions is close to the probabilistic formulation. In this case the Hamilton--Jacobi equation is considered in the minimax/viscosity sense, whereas the flow of probabilities is determined by the probability on the set of solutions of the differential inclusion associated with the Hamilton-Jacobi equation such that those solutions are viable in the graph of the minimax solution. The main result of the paper is the convergence (up to subsequence) of the solutions of the mean field games to the minimax  solution of a deterministic mean field game in the case when the underlying dynamics converge to the deterministic evolution.\end{abstract}

{\small \noindent{\bf Keywords:} mean field games, deterministic limit, minimax solutions.

\noindent{\bf MSC codes:} 91A23, 91A13, 49L20, 3E20}

\section{Introduction}
The  mean field game theory initiated independently by Lasry and Lions  (see \cite{Lions01}--\cite{lions_frans}) and Huang, Malham\'{e},
and Caines  (see \cite{Huang1}--\cite{Huang4}) provides a way to describe a control system
with a large number of independent players by studying the limit system when the number of players tends to infinity. It is assumed that the players are identical and the dynamics and reward of each player   depend on her state, her control and the distribution  of all the players.

There are several approaches to analysis of mean field games. A first way consists in studying the coupled system made of a Hamilton-Jacobi equation for the value function and a Kolmogorov equation for the law of probabilities. This approach was used for  stochastic mean field games (see  \cite{bensoussan_book},  \cite{Gomes}, \cite{Lions}, \cite{lions_frans} and reference therein), for deterministic mean field games (see \cite{Lions}, \cite{lions_frans}) and for the mean field games with a finite number of states (see \cite{Gomes_finite_state}).
The most general case when the dynamics of each player is controlled by a nonlinear Markov process is considered in the framework of this approach in \cite{kolokoltsev_mean_field}--\cite{Kolokoltstov_dgaa}.

The second approach is based on a probabilistic formulation of mean field games. In this case the Hamilton-Jacobi equation is replaced with the optimization problem whereas the probability law is given by the distribution of the states of the players choosing their optimal strategies. This approach was developed for stochastic mean field games in \cite{Carmona_1}--\cite{Carmona_3}, \cite{Lacker} and for deterministic mean field games in \cite{Gomes_extended}. Note that the probabilistic approach is convenient for the convergence problem. It is used in~\cite{Fischer},~\cite{Lacker} to analyse the limit behaviour of symmetric equilibria for $N$-player games.

The concept of minimax solutions proposed in \cite{averboukh_mfg} for the deterministic case is close to the probabilistic approach. In this case the Hamilton--Jacobi equation is considered in the minimax/viscosity sense, whereas the flow of probabilities is determined by the probability on the set of solutions of the differential inclusion associated with the Hamilton-Jacobi equation such that those solutions are viable in the graph of the minimax solution. The notion of minimax solution to mean field game goes back to the  minimax solution introduced for Hamilton-Jacobi PDE by Subbotin~\cite{Subb_book}. The definitions of minimax and viscosity solutions for Hamilton--Jacobi PDEs are equivalent. The minimax solutions  deal primarily with the case when the Hamiltonian enjoys sublinear growth.

The third approach developed in \cite{bensoussan_book}, \cite{bensoussan_master}, \cite{mean_field_convergence}, \cite{Carmona_master}, \cite{Gomes_extended}, \cite{Kolokoltstov_common} relies on consideration of the mean field game as a dynamical system in the space of probabilities. In the framework of this approach the study of mean field game is performed by examining the master equation that is a Hamilton--Jacobi equation involving probability as a state variable. This approach permits to prove that a limit of feedback Nash equilibria when the number of player tends to infinity is a solution of a mean field game \cite{mean_field_convergence}.

The paper is concerned with the deterministic limit of the mean field games. The deterministic limit of the second order mean field games was studied by Lasry and Lions in \cite{Lions}, and Lions in \cite{lions_frans} (see also \cite{cardal_lions}).
We consider the case when the dynamics are controlled by nonlinear Markov processes as exposed in  monograph~\cite{Kol_book}. These systems include processes described by
stochastic differential equations or continuous time Markov chains. 

  %The approach based on probabilistic approach is used.
%In those works the deterministic limit of mean field game systems was studied in the case when the dynamics of the original games are governed by stochastic differential equations. Additionally some smoothness conditions are imposed.

%In the paper extending the definition of \cite{Gomes}, \cite{Gomes_extended}, \cite{Carmona_1}--\cite{Carmona_3} we introduce the concept of solution of mean field game system in the probabilistic sense the case when the dynamics is governed by nonlinear Markov process. It is proved that the classical solution of the mean field system studied with dynamics given by nonlinear Markov process is a solution in the probabilistic sense.

The main result  of the paper is the convergence (up to subsequence) of the solutions of mean field games  with nonlocal coupling to the minimax solution of the first order mean field game  in the case when the corresponding dynamics governed by nonlinear Markov processes converge to the deterministic evolution.  As a corollary we obtain the equivalence of the minimax solutions and the solutions in the probabilistic sense for first-order mean field games.

The results of the paper are based on the relative compactness of solutions of mean field games. This idea was used in several papers to construct  solutions of  second-order mean field games with unbounded coefficients (see  \cite{carmona_lacker_del_common}, \cite{lacker_existence}). The  difference between those papers and this one is as follows. In \cite{carmona_lacker_del_common}, \cite{lacker_existence} the sequence of mean field games is constructed by the given mean field game with unbounded coefficients. Thus, the underlying probability spaces of all games coincide. Moreover, the mentioned papers are concerned only with the case of second-order mean field games. In the present paper we consider the mean field games with dynamics given by nonlinear Markov processes. %Thus, the underlying probability spaces for the original mean field games and for the limiting one do not coincide. 
Thus, the probability spaces for original mean field games include the sets of c\`{a}dl\`{a}g functions; whereas the solution of the limiting mean field game is determined by the probability on the set of continuous functions. To overcome these difficulties we introduce auxiliary stochastic processes governed by ODE under stochastic control policies borrowed  from solutions of the original mean field games. The main idea of the proof of the main result is to evaluate the distance between original and auxiliary processes and to show the relative compactness of the auxiliary stochastic processes.

The outline of the paper is as follows. Section~\ref{sect_prelim} provides main notations. In this section we introduce the definition of a solution to a mean field game system in the probabilistic sense for the case when the dynamics of each player is governed by a nonlinear Markov process. Additionally, we recall the definition of a minimax solution to a first-order mean field game system. In Section~\ref{sect_main}  the main result of the paper is formulated. In Section~\ref{sect_bounds} we introduce an auxiliary stochastic processes and obtain  the bounds for moments of the original and auxiliary stochastic processes. The distance between original and auxiliary stochastic processes is estimated in Section~\ref{sect_distance}. Finally, Section~\ref{sect_limit} provides the proof of the main result. It is based on the relative compactness of the auxiliary stochastic processes and on the analysis of the limiting probability on the set of trajectories. 

\section{Preliminaries}\label{sect_prelim}
\subsection{Main notations}
If $A$ is a Banach space, then denote by $\mathcal{P}(A)$ the set of probabilities on $A$. Further let  $\mathcal{P}^2(A)$ denote the set of probabilities $m$ on $A$ satisfying
$$\varsigma^2(m)\triangleq \int_A \|a\|^2m(da)<\infty.$$
Denote $$\varsigma(m)\triangleq \sqrt{\varsigma^2(m)}.$$
We consider the 2-Wasserstein distance between $m_1,m_2\in\mathcal{P}^2(A)$ i.e.
$$W_2(m_1,m_2)\triangleq \left(\inf\left\{\int_{A\times A}\|a_1-a_2\|^2\pi(d(a_1,a_2)):\pi\in \Pi(m_1,m_2)\right\}\right)^{1/2}. $$ Here $$\Pi(m_1,m_2)\triangleq \{\pi\in \mathcal{P}^2(A\times A):\pi(\Gamma\times A)=m_1(\Gamma),\ \ \pi(A\times \Gamma)=m_2(\Gamma)\}.$$

The space $\mathcal{P}^2(\mathcal{A})$  with the metric $W_2$ is a complete metric space \cite{ambrosio_gradient_flows}, \cite{Villani}.

A measurable function $\mu:[0,T]\rightarrow\mathcal{P}^2(\mathbb{R}^d)$ is called a flow of probabilities. Below we denote the set of flows of probabilities by $\mathcal{M}$.

If $(\Omega^1,\mathcal{F}^1)$, $(\Omega^2,\mathcal{F}^2)$ are measurable spaces, $h:\Omega^1\rightarrow \Omega^2$ is measurable, $m$ is a measure on $\Omega^1$, then $h_\# m$ denotes the push-forward measure on $\Omega^2$ i.e. for $\Gamma\in \mathcal{F}^2$,
$$(h_\# m)(\Gamma)=m(h^{-1}(\Gamma)). $$

We consider the space $\mathbb{R}^{d+1}$ as the set of pairs $w=(x,z)$, where $x\in\mathbb{R}^d$, $z\in\mathbb{R}$.
To simplify the notation put
$$\mathcal{C}\triangleq C([0,T],\mathbb{R}^{d+1}).$$
 Denote by $e_t$ the following projection from $\mathcal{C}$ onto $\mathbb{R}^d$ $$e_t(x(\cdot),z(\cdot))\triangleq x(t).$$

Note that if $\chi_1,\chi_2\in\mathcal{P}^2(\mathcal{C})$, then \begin{equation}\label{wass_estima}
W_2({e_t}_{\#}\chi_1,{e_t}_{\#}\chi_2)\leq W_2(\chi_1,\chi_2).
\end{equation}
Indeed , if $\pi\in \Pi(\chi_1,\chi_2)$, then
\begin{equation*}\begin{split}
\int_{\mathcal{C}\times \mathcal{C}}\|e_t(w_1(\cdot))&-e_t(w_2(\cdot))\|^2\pi(d(w_1(\cdot),w_2(\cdot)))\\&\leq \int_{\mathcal{C}\times \mathcal{C}}\left(\sup_{\tau\in [0,T]}\|w_1(\tau)-w_2(\tau)\|\right)^2\pi(d(w_1(\cdot),w_2(\cdot)). \end{split}\end{equation*}
Moreover,
\begin{equation*}
W_2^2({e_t}_{\#}\chi_1,{e_t}_{\#}\chi_2)\leq \inf_{\pi\in \Pi(\chi_1,\chi_2)}\int_{\mathcal{C}\times \mathcal{C}}\|e_t(w_1(\cdot)-e_t(w_2(\cdot))\|^2\pi(d(w_1(\cdot),w_2(\cdot)))
%\end{split}
\end{equation*}
Thus, (\ref{wass_estima}) holds true.

\subsection{Stochastic mean field games}
For a parameter $n\in\mathbb{N}$, a probability $m\in\mathcal{P}^2(\mathbb{R}^d)$, and a constant control $u\in U$ let $L^n_t[m,u]$ be a generator of L\'{e}vy-Khintchine type, i.e.
\begin{equation}\label{L_n_def}
\begin{split}L^n_t[m,u]\varphi(x)=\frac{1}{2}&\langle G^n(t,x,m,u)\nabla,\nabla\rangle\varphi(x)+\langle f^n(t,x,m,u),\nabla\rangle \varphi(x)\\&+\int_{\mathbb{R}^d}[\varphi(x+y)-\varphi(x)-\langle y,\nabla \varphi(x)\rangle\mathbf{1}_{B_1}(y)]\nu^n(t,x,m,u,dy). \end{split}\end{equation} Here $U$ is a control space, $G^n(t,x,m,u)$ is a nonnegative symmetric matrix, $B_1$ denotes the unit ball  centered at the origin,   $\nu^n(t,x,m,u,\cdot)$ is a measure on $\mathbb{R}^d$ such that
$\nu^n(t,x,m,u,\{0\})=0$ and
$$\int_{\mathbb{R}^d} \min\{1,y^2\}\nu^n(t,x,m,u,dy)<\infty. $$

Let $\mathcal{D}^n\subset C^2(\mathbb{R}^d)$ be a Banach subspace such that $C_b^2(\mathbb{R}^d)\subset \mathcal{D}^n$ and, for any $\xi\in\mathbb{R}^d$, the functions $x\mapsto\langle \xi,x\rangle $, $x\mapsto \|x-\xi\|^2$ belong to  $\mathcal{D}^n$.

Note that the dynamics of second order mean field games can be expressed in the form (\ref{L_n_def}) with $\nu^n=0$ \cite{kolokoltsev_mean_field}. Analogously, pure jump processes are specified by the generator of  form (\ref{L_n_def}) with $f^n=0$, $G^n=0$ \cite{kolokoltsev_mean_field}.

Extending definition introduced in \cite[p. 135]{fleming_soner} to the mean field game case we say that a 6-tuple $(\Omega,\mathcal{F},\{\mathcal{F}_t\}_{t\in[s,T]},P,u,Y)$ is a control system on $[s,T]$ admissible  for a generator $L^n$ and a flow of probabilities $ \zeta$,  if
\begin{enumerate}
\item $(\Omega,\mathcal{F},\{\mathcal{F}_t\}_{t\in[s,T]},P)$ is a filtered probability space;
\item $u$ is a $\{\mathcal{F}_t\}_{t\in[s,T]}$-progressively measurable process with values in $U$; $Y$ is a $\{\mathcal{F}_t\}_{t\in[s,T]}$-adapted c\`{a}dl\`{a}g process with values in $\mathbb{R}^d$;
\item for any $\varphi\in\mathcal{D}^n$, the process
$$\varphi(Y(t))-\int_s^t L^n_\tau[\zeta[\tau],u(\tau)]\varphi(Y(\tau))d\tau $$ is a $\{\mathcal{F}_t\}_{t\in[s,T]}$-martingale.
\end{enumerate}
Below we assume the following condition.
\begin{list}{(A\arabic{tmp})}{\usecounter{tmp} \setcounter{tmp}{-1}}
\item The generator $L^n$ is such that, for any flow of probabilities $\zeta$, any measurable deterministic function $v:[s,T]\rightarrow U$ and any $(s,\xi)\in[0,T]\times\rd$, there exist a filtered probability space $(\Omega,\mathcal{F},\{\mathcal{F}_t\}_{t\in[s,T]},P)$ and a process $Y$ such that $Y(s)=\xi$ $P$-a.s. and $(\Omega,\mathcal{F},\{\mathcal{F}_t\}_{t\in[s,T]},P,v,Y)$ is admissible for $L^n$ and $\zeta$.
\end{list}
\begin{remark}
  The conditions on $G^n$, $f^n$ and $\nu^n$ which guarantee  assumption (A0) can be derived from \cite[Theorem 5.4.2, Theorem 5.5.1]{Kol_markov}. It suffices to assume that 
  \begin{itemize}
  \item the function
  \begin{multline*}(t,x,u,\xi)\mapsto \frac{1}{2}\langle G^n(t,x,m,u)\xi,\xi\rangle-i
  \langle f^n(t,x,m,u),\xi\rangle\\+\int_{\mathbb{R}^d}(1-e^{i\langle \xi,y\rangle}+i\mathbf{1}_{B_1}(y)\langle\xi,y\rangle)\nu^n(t,x,m,u,dy) \end{multline*} is continuous (here $i$ stays for imaginary unit);
  \item for any $C>0$,
  \begin{multline*}\sup_{t\in[0,T],x\in\mathbb{R}^d,m\in\mathcal{P}^2(\mathbb{R}^d),\varsigma(m)<C,u\in U}\Bigl(\frac{\|G^n(t,x,m,u)\|}{1+\|x\|^2}+ \frac{\|f^n(t,x,m,v(t))\|}{1+\|x\|}\\+\frac{\int_{B_1}y^2\nu^n(t,x,m,u,dy)}{1+\|x\|^2}+ \int_{\mathbb{R}^d\setminus B_1}\nu^n(t,x,m,u,dy) \Bigr)<\infty,\end{multline*}
 \item $$\sup_{t\in [0,T],x\in\mathbb{R}^d,m\in\mathcal{P}^2(\mathbb{R}^d),\varsigma(m)<C,u\in U}\int_{\mathbb{R}^d\setminus B_1}\ln(\|y\|)\nu^n(t,x,m,u,dy) <\infty.$$

\end{itemize}
\end{remark}
Given the generator $L^n$, we consider the following mean field game system \cite{kolokoltsev_mean_field}:
\begin{equation}\label{HJ_L_n}
  \frac{\partial V(t,x)}{\partial t}+\max_{u\in U}(L^n_t[\zeta[t],u]V(t,x)+g(t,x,\zeta[t],u))=0,
\end{equation}
\begin{equation}\label{kynetic_L_n}
\frac{d}{dt}\langle \phi,\zeta[t]\rangle=\langle L^n_t[\zeta[t],u^*(t,\cdot,\zeta[t], V(t,\cdot))]\phi,\zeta[t]\rangle\ \ \forall \phi\in\mathcal{D}^n,
\end{equation}
\begin{equation}\label{cond_L_n}
  V(T,x)=\sigma(x,\zeta[T]),\ \ \zeta[0]=m_0^n.
\end{equation}
Here $u^*(t,x,m,\varphi):[0,T]\times\rd\times\mathcal{P}(\mathbb{R}^d)\times\mathcal{D}^n\rightarrow U$ is such that
$$u^*(t,x,m,\varphi)={\rm argmax}\{L^n_t[m,u]\varphi(x)+g(t,x,m,u):u\in U\}. $$

We %extend the definition of the solution of system (\ref{HJ_L_n})--(\ref{cond_L_n}) and
consider the probabilistic approach to mean field games. This approach was developed for  second-order mean field games  in \cite{Carmona_1}--\cite{Carmona_3}, \cite{lacker_existence} and for deterministic mean field games in \cite{Gomes}. The following definition extends the definitions of the mentioned papers to the case when the mean field game is associated with the nonlinear Markov processes.
\begin{definition}\label{def_probability} We say that  a function $V^n:[0,T]\times\rd\rightarrow \mathbb{R}$ and a flow of probabilities $\zeta^n$ solve system (\ref{HJ_L_n})--(\ref{cond_L_n}) in the probabilistic sense, if the following conditions hold true.
\begin{enumerate}[label=(\roman*)] 
	\item $\zeta^n[0]=m_0^n$.
	\item $V^n(s,\xi)$ is a value  of  the optimal control problem 
	\begin{equation}\label{objective_function}\text{maximize }\mathbb{E}\left[\sigma(Y(T),\zeta^n[T])+ \int_s^Tg(\tau,Y(\tau),\zeta^n[\tau],u(\tau))d\tau\right] \end{equation}
	over  all control systems $(\Omega,\mathcal{F},\{\mathcal{F}_t\}_{t\in[s,T]},P,u,Y)$ admissible  for the generator $L^n$ and the flow of probabilities $ \zeta^n$ such that $Y(s)=\xi$ $P$-a.s. Here $\mathbb{E}$ denotes the expectation according to $P$.
	\item  There exists a  control system on $[0,T]$ $(\Omega^n,\mathcal{F}^n,\{\mathcal{F}^n_t\}_{t\in[0,T]},P^n,u^n,Y^n)$  admissible for $L^n$ and $\zeta^n$  such that
      \begin{equation}\label{V_n_process}
       V^n(t,Y^n(t))+ \int_0^tg(\tau,Y^n(\tau),\zeta^n[\tau],u^n(\tau))d\tau\end{equation}
       is $\{\mathcal{F}^n_t\}_{t=0}^T$-martingale and
      \begin{equation}\label{V_n_law}\zeta^n[t]={\rm Law}(Y^n(t)),
      \end{equation}

      %In (\ref{V_n_process})  
\end{enumerate}

\end{definition}

\begin{remark} Condition (ii) of Definition \ref{def_probability} means that
	\begin{itemize}
	\item
	 for any control systems $(\Omega,\mathcal{F},\{\mathcal{F}_t\}_{t\in[s,T]},P,u,Y)$ admissible  for the generator $L^n$ and the flow of probabilities $ \zeta^n$ such that $Y(s)=\xi$ $P$-a.s., we have that
	 $$V(s,\xi)\geq \mathbb{E}\left[\sigma(Y(T),\zeta^n[T])+ \int_s^Tg(\tau,Y(\tau),\zeta^n[\tau],u(\tau))d\tau\right]; $$
	 \item for any $\delta>0$, one can find a control systems $(\Omega^{n,\delta}_{s,\xi},\mathcal{F}^{n,\delta}_{s,\xi},\{\mathcal{F}^{n,\delta}_{s,\xi,t}\}_{t\in[s,T]},P^{n,\delta}_{s,\xi},u^{n,\delta}_{s,\xi}, Y^{n,\delta}_{s,\xi})$ admissible  for the generator $L^n$ and the flow of probabilities $ \zeta^n$ such that $Y^{n,\delta}_{s,\xi}(s)=\xi$ $P^{n,\delta}_{s,\xi}$-a.s. satisfying the following condition:
	 	 $$V(s,\xi)\leq \mathbb{E}^{n,\delta}_{s,\xi}\left[\sigma(Y^{n,\delta}_{s,\xi}(T),\zeta^n[T])+ \int_s^Tg(\tau,Y^{n,\delta}_{s,\xi}(\tau),\zeta^n[\tau],u^{n,\delta}_{s,\xi}(\tau))d\tau\right]+\delta. $$
\end{itemize}
Here $\mathbb{E}$ (respectively, $\mathbb{E}^{n,\delta}_{s,\xi}$) stands for the  expectation according to the probability $P$ (respectively $P^{n,\delta}_{s,\xi}$).

      The control system $(\Omega^n,\mathcal{F}^n,\{\mathcal{F}^n_t\}_{t\in[0,T]},P^n,u^n,Y^n)$ is $\delta$-optimal for the considered optimization problem.
\end{remark}

The link between solution in probabilistic approach and classical solutions to mean field games is given in the following proposition.

\begin{proposition}\label{pl_consistence}Assume that $(V^n,\zeta^n)$ is a classical solution to system (\ref{HJ_L_n})--(\ref{cond_L_n}) and
 there exists a solution to the martingale problem specified by the operator
$ L^n_t[\zeta^n[t],u^*(t,\cdot,\zeta^n[t],V^n(t,\cdot))]:\mathcal{D}^n\rightarrow {C}(\mathbb{R}^d)$.
Then the pair $(V^n,\zeta^n)$ solves (\ref{HJ_L_n})--(\ref{cond_L_n}) in the probabilistic sense.
\end{proposition}
%\begin{remark}Note that the condition of the Proposition is weaker conditions for the existence theorem of (\ref{HJ_L_n})--(\ref{cond_L_n}) (see \cite[Theorem 2.2]{kolokoltsev_mean_field}).
%\end{remark}
\begin{proof}
The proof is close to the proof of the verification theorem \cite[Theorem 8.1]{fleming_soner}.

First, notice that the condition (i) of Definition \ref{def_probability} is fulfilled.

Since $V^n$ is a solution to equation (\ref{HJ_L_n}) we have that  $V^n(s,\xi)$ is a value  for the optimization problem for the objective functional (\ref{objective_function}) over  6-tuples $(\Omega,\mathcal{F},\{\mathcal{F}_t\}_{t\in[s,T]},P,u,Y)$ those are control systems admissible  for the generator $L^n$ and the flow of probabilities $ \zeta^n$ such that $Y(s)=\xi$ $P$-a.s. 
This proves condition (ii). 

By condition of the Proposition there exist a probability space $(\Omega^n,\mathcal{F}^n,\{\mathcal{F}^n_t\}_{t\in [0,T]},P^n)$ and a stochastic process $Y^n(\cdot)$ such that, for any $\varphi\in\mathcal{D}^n$,
$$\varphi(Y^n(t))-\int_0^t L^n_\tau[\zeta^n[\tau],u^*(\tau,\cdot,\zeta^n[\tau],V^n(\tau,\cdot))]\varphi(Y^n(\tau))d\tau $$ is a $\{\mathcal{F}_t^n\}_{t\in[0,T]}$-martingale and ${\rm Law}(Y^n(0))=\zeta^n[0]$. Using (\ref{kynetic_L_n}), we obtain that
\begin{equation}\label{law_eq}
{\rm Law}(Y^n(t))=\zeta^n[t].
\end{equation}

Put \begin{equation}\label{intro:u_n_proposition}
u^n(t,\omega)\triangleq u^*(t,Y^n(t,\omega),\zeta^n[t],V^n(t,\cdot)).
\end{equation} Notice that
the 6-tuple $(\Omega^n,\mathcal{F}^n,\{\mathcal{F}^n_t\}_{t\in [0,T]},P^n,u^n,Y^n)$ is admissible for the generator $L^n$ and the flow of probabilities $\zeta^n$.
Furthermore, dynamic programming arguments give that, if $u_n$ is defined by (\ref{intro:u_n_proposition}), then (\ref{V_n_process}) is $\{\mathcal{F}^n_t\}_{t\in [0,T]}$-martingale. This and (\ref{law_eq}) yield condition (iii) of Definition \ref{def_probability}.

Therefore,  the pair $(V^n,\zeta^n)$ solves (\ref{HJ_L_n})--(\ref{cond_L_n}) in the probabilistic sense.
\end{proof}

\subsection{Deterministic mean field game}
Consider the deterministic mean field game system
\begin{equation}\label{HJ}
  \frac{\partial V}{\partial t}+H(t,x,\mu[t],\nabla V)=0,\ \ V(T,x)=\sigma(x,\mu[T]),
\end{equation}
\begin{equation}\label{kinetic}
  \frac{d}{dt}\mu[t]=\left\langle\frac{\partial H(t,x,\mu[t],\nabla V)}{\partial p},\nabla\right\rangle\mu[t], \ \ \mu[0]=m_0.
\end{equation} Here $p$ stands for the 4-th coordinate  in
\begin{equation*}
H(t,x,m,p)=\max\{\langle p,f(t,x,m,u)\rangle+g(t,x,m,u):u\in U\}.
\end{equation*} System (\ref{HJ}), (\ref{kinetic}) can be rewritten in the  form (\ref{HJ_L_n})--(\ref{cond_L_n}) for the generator
$$L^*_t[\mu,u]\varphi(x)=\langle f(t,x,\mu,u),\nabla \varphi(x)\rangle. $$ 

Note that, for a given flow of probabilities $\mu$, equation (\ref{HJ}) provides the value function of the following optimization problem:
\begin{equation}\label{mfg_def:payoff}
  \mbox{maximize }\sigma(x(T),\mu(T))+\int_{s}^Tg(t,x(t),u(t),\mu(t))dt
\end{equation}
\begin{equation}\label{mfg_def:conditions}
  \mbox{subject to }\frac{d}{dt}x(t)=f(t,x(t),\mu(t),u(t)),\ \ x(s)=\xi,\ \ x(t)\in\mathbb{R}^d,\ \ u(t)\in U.
\end{equation}

The  solution for optimization problem (\ref{mfg_def:payoff}), (\ref{mfg_def:conditions}) does not exist in the general case~\cite{Warga}. To guarantee the existence of the optimal trajectory we consider the relaxation of this problem based on differential inclusion.

Given a flow of probabilities $\mu$,  denote the set of solutions of the differential inclusion
$$(\dot{x}(t),\dot{z}(t))\in{\rm co}\{(f(t,x,\mu[t],u),g(t,x,\mu[t],u)):u\in U\},\ \ x(s)=\xi,\ \ z(s)=0$$
by ${\rm Sol}(\mu,s,\xi)$. Here $\mathrm{co}$ denotes the convex hull. Additionally, put
$${\rm Sol}_0(\mu)\triangleq \bigcup_{\xi\in\mathbb{R}^d}{\rm Sol}(\mu,0,\xi). $$

Under assumptions made in this paper (see conditions (A1)--(A5) below) the problem
\begin{equation}\label{mfg_def:relaxed_payoff}
  \mbox{maximize }\sigma(x(T),\mu(T))+z(T)
\end{equation}
\begin{equation}\label{mfg_def:relax_set}
  \mbox{subject to }(x(\cdot),z(\cdot))\in \mathrm{Sol}(\mu,s,\xi)
\end{equation} is a relaxation of original problem (\ref{mfg_def:payoff}), (\ref{mfg_def:conditions}). Moreover, any element of $\mathrm{Sol}(\mu,s,\xi)$ can be approximated by trajectories generated by usual controls. Denote by  $\mathcal{U}$ the set of measurable functions $v:[0,T]\rightarrow U$. If $s\in [0,T]$, $\xi\in\mathbb{R}^d$, $v\in \mathcal{U}$, then let  $x[\cdot,\mu,s,\xi,v]$ be a solution of the initial value problem
$$\dot{x}(t)=f(t,x(t),\mu[t],v(t)),\ \ x(s)=\xi. $$ Moreover, put
\begin{equation}\label{z_def}
z[t,\mu,s,\xi,v]=\int_s^t g(\tau,x[\tau,\mu,s,\xi,v],\mu[\tau],v(\tau))d\tau.
\end{equation}
 By \cite[Theorem VI.3.1]{Warga} we have that
\begin{equation*}%\label{closure_sol}
  {\rm Sol}(\mu,s,\xi)={\rm cl}\{(x[\cdot,\mu,s,\xi,v],z[\cdot,\mu,s,\xi,v]):v\mbox{ is measurable}\}.
\end{equation*} Here closure is taken in the space $\mathcal{C}=C([0,T],\mathbb{R}^{d+1})$. Moreover, Gronwall's inequality implies that
\begin{equation}\label{x_control_estima}
\|x[t,\mu,s,\xi,v]\|\leq \left(\|\xi\|+MT+MT\sup_t\varsigma(\mu[t])\right) e^{MT}.
\end{equation} 

Note that the relaxation based on differential inclusions is equivalent to the approaches based on measure-valued controls or on control measures.  This means that $(x(\cdot),z(\cdot))\in\mathrm{Sol}(\mu,s,\xi)$ if and only if there exists a measure-valued control $\gamma:[0,T]\rightarrow\mathcal{P}(U) $ such that
$$\dot{x}(t)=\int_Uf(t,x(t),\mu(t),u)\gamma(t,du),\ \ x(s)=\xi,$$ $$\dot{z}(t)=\int_Ug(t,x(t),\mu(t),u)\gamma(t,du),\ \ z(s)=0. $$ Here $\mathcal{P}(U)$ denotes the set of all probabilities on $U$. The equivalence between approaches based on measure-valued controls and on control measures can be find, for example, in \cite{averboukh_mfg}.

We use minimax solutions first proposed in \cite{averboukh_mfg}. Note that the first order mean field games can be considered within the framework of probabilistic approach. The advantage of the minimax solutions to mean field games is that they do not depend on the choice of the probability space and the representation of the Hamiltonian $H$.  The link between minimax solutions and solution in the probabilistic sense is given in Corollary 2 below.

The definition of a minimax solution to a deterministic mean field game involves the definition of a minimax solution to a Hamilton--Jacobi PDE.
There exist several (equivalent) definitions of a minimax solution \cite{Subb_book} to a Hamilton--Jacobi PDE. They are based either on viability theory or on  nonsmooth analysis. For our purposes the definition involving optimization problem is more useful.

For $t\in [0,T]$, $x\in\mathbb{R}^d$, $m\in\mathcal{P}^2(\mathbb{R}^d)$ and $y\in \mathbb{R}^d$, put
$$H^*(t,x,m,y)=\sup_{p\in\mathbb{R}^n}[\langle  y,p\rangle-H(t,x,m,p)],$$ $$ \mathcal{H}(t,x,m)=\left\{\xi\in\mathbb{R}^n:H^*(t,x,m,\xi)<\infty\right\}. $$

The function $V$ is a \textit{minimax solution} to equation (\ref{HJ}) for the given flow of probabilities $\mu$, if
$V(s,\xi)$ is a value of the optimization problem
\begin{equation}\label{max_hamiltonian}
\mbox{maximize }\bigl[\sigma(x(T),\mu[T])+z(T)-z(s)\bigr]
\end{equation}
\begin{equation}\label{constrain_hamiltonian}
\mbox{subject to }\dot{x}(t)\in\mathcal{H}(t,x(t),\mu[t]),\ \ \dot{z}(t)\leq -H^*(t,x(t),\mu[t],\dot{x}(t)), \ x(s)=\xi.
\end{equation}

This is equivalent to the following property: $V(s,\xi)$ is a value of the optimization problem  (\ref{mfg_def:relaxed_payoff}), (\ref{mfg_def:relax_set}).

Denote
\begin{equation*}\begin{split}\mathcal{S}[V,\mu]\triangleq\{(x(\cdot)&,z(\cdot)):\dot{x}(t)\in \mathcal{H}(t,x(t),\mu[t]),\\
\dot{z}(t)&=-H^*(t,x(t),\mu[t],\dot{x}(t)), \ \ z(t)=V(t,x(t))\}. \end{split}\end{equation*}

\begin{definition}\label{def_solution}
We say that the pair $(V^*,\mu^*)\in C([0,T]\times\rd)\times\mathcal{M} $  is a minimax solution to  system (\ref{HJ}), (\ref{kinetic}), if
\renewcommand{\theenumi}{\arabic{enumi}}
\begin{enumerate}
  \item $V^*$ is a minimax solution of  equation (\ref{HJ});
  \item ${\mu}^*[0]=m_0 $
  \item there exists a probability $\chi$ on $\mathcal{C}$ such that $\mu^*[t]=e_t{}_\#\chi$ and $\mathrm{supp}(\chi)\subset \mathcal{S}[V^*,\mu^*]$.
 \end{enumerate}
\end{definition}
Note that the definition of the minimax solution to mean field game does not depend on the representation of the Hamiltonian.

The embedding of the set of optimal trajectories of  problem (\ref{mfg_def:relaxed_payoff}), (\ref{mfg_def:relax_set}) into the set $\mathcal{S}[V,\mu]$ implies the following statement.
\begin{proposition}\label{pl_minimax_suff}Assume that the pair $(V^*,\mu^*)\in C([0,T]\times\rd)\times\mathcal{M} $  satisfies the following conditions:
\begin{enumerate}
\item  for any $(s,\xi)\in[0,T]\times\rd$, $V^*(s,\xi)$ is a value  of the optimization problem~(\ref{mfg_def:relaxed_payoff}),~(\ref{mfg_def:relax_set});
  \item $\mu^*[0]=m_0$;
  \item there exists a probability $\chi\in\mathcal{P}^2(\mathcal{C})$ such that $\mu^*[t]=e_t{}_\#\chi$ and \begin{equation}\label{optimal_trajectories}\begin{split}{\rm supp}(\chi)\subset\{&(x(\cdot),z(\cdot))\in {\rm Sol}_0(\mu):\\&V^*(s,x(s))=\sigma(x(T),\mu^*[T])+z(T)-z(s), \ \ s\in[0,T]\}. \end{split}\end{equation}
\end{enumerate} Then $(V^*,\mu^*)$ is a minimax solution to system (\ref{HJ}), (\ref{kinetic}).
\end{proposition}

\section{Main result}\label{sect_main}

To simplify the notations put
$$\Sigma^n(t,x,m,u)\triangleq \sum_{i=1}^d G_{ii}^n(t,x,m,u)+\int_{\mathbb{R}^d}\|y\|^2\nu^n(t,x,m,u,dy), $$
$$b^n(t,x,m,u)\triangleq f^n(t,x,m,u)+\int_{\mathbb{R}^d\setminus B_1} y\nu^n(t,x,m,u,dy). $$

Below, if $\phi$ is a function of $k$ vectors and takes values in $\mathbb{R}$, then we assume that $L^n_t$ (respectively, $L^*_t$) acts only on the first variable. Using this convention, we get that, if $s,t\in [0,T]$, $s<t$, $\zeta$ is a flow of probabilities, $(\Omega,\mathcal{F},\{\mathcal{F}_t\}_{t\in [s,T]},P,X,u)$ is admissible for $L^n$ and $\zeta$,   $\eta_1,\ldots,\eta_k:\Omega\rightarrow\mathbb{R}^d$ are measurable w.r.t. $\mathcal{F}_s$,  $\phi:(\mathbb{R}^d)^{k+1}\rightarrow\mathbb{R}$ is such that $\phi(\cdot,x_1,\ldots,x_k)\in\mathcal{D}^n$, then
\begin{equation}\label{phi_s_t_martingale_L_n}\begin{split}
\mathbb{E}\phi(X(t),\eta_1,\ldots,\eta_k)=\mathbb{E}\phi&(X(s),\eta_1,\ldots,\eta_k)\\&+ \mathbb{E}\int_s^t L^n_\tau[\zeta[\tau],u(\tau)]\phi(X(\tau),\eta_1,\ldots,\eta_k)d\tau.
\end{split}\end{equation} Here $\mathbb{E}$ denote the expectation according to $P$. 

Analogous formula is fulfilled for the generator $L^*$ i.e. if $\mu$ is a flow of probabilities, $(\Omega,\mathcal{F},\{\mathcal{F}_t\}_{t\in [0,T]},P,X,u)$ is such that
$$\frac{d}{dt}X(t,\omega)=f(t,X(t,\omega),\mu[t],u(t,\omega)), $$ then \begin{equation}\label{phi_s_t_martingale_L_star}
\mathbb{E}\phi(X(t),\eta_1,\ldots,\eta_k)=\mathbb{E}\phi(X(s),\eta_1,\ldots,\eta_k)+ \mathbb{E}\int_s^t L^*_\tau\phi(X(\tau),\eta_1,\ldots,\eta_k)d\tau.
\end{equation}

Note that if $q(x_1,x_2)=\|x_1-x_2\|^2$, then
\begin{equation}\label{generator_qaud}
  (L^n_t[m,u]q)(x_1,x_2)= \Sigma^n(t,x_1,m,u)+2\langle b^n(t,x_1,m,u),x_1-x_2\rangle.
\end{equation}
%\begin{equation}\label{generator_qaud_det}
%  L^*_t[m,u]q_\xi(x)= 2\langle f(t,x,m,u),x-\xi\rangle,
%\end{equation}
Moreover, if $\lambda(x_1,x_2,x_3)=\langle x_1-x_2,x_3\rangle$,
\begin{equation}\label{generator_lin}
(L^n_t[m,u]\lambda)(x_1,x_2,x_3)=\langle  b^n(t,x_1,m,u),x_3\rangle,
\end{equation}
\begin{equation}\label{generator_det_lin}
(L^*_t[m,u]\lambda)(x_1,x_2,x_3)=\langle  f(t,x_1,m,u),x_3\rangle.
\end{equation}

 We put the following assumptions on the set $U$, functions $f^n$, $f$, $g$ and $\sigma$.

\begin{list}{(A\arabic{tmp})}{\usecounter{tmp}}
  \item $U$ is compact;
  \item There exists a function $\alpha(\cdot)$  such that $\alpha(\delta)\rightarrow 0$ as $\delta\rightarrow 0$ and, for any $t',t''\in[0,T]$, $x\in\mathbb{R}^d$, $m\in \mathcal{P}^2(\mathbb{R}^d)$, $u\in U$,
      $$\|f(t',x,m,u)-f(t'',x,m,u)\|\leq \alpha(t'-t''); $$
     % $$\|b^n(t',x,m,u)-b^n(t'',x,m,u)\|\leq \alpha(t'-t''); $$
  \item\label{cond_lipschitz} there exists a constant $K$ such that, for any $t\in [0,T]$, $x',x''\in\mathbb{R}^d$, $u\in U$, $m',m''\in \mathcal{P}^2(\mathbb{R}^d)$,
  $$\|f(t,x',m',u)-f(t,x'',m'',u)\|\leq K(\|x'-x''\|+W_2(m',m'')), $$
  %$$\|b^n(t,x',m',u)-b^n(t,x'',m'',u)\|\\\leq K(\|x'-x''\|+W_2(m',m'')), $$
  \item\label{cond_growth} there exists a constant $M$ such that,  for any $t\in [0,T]$, $x\in\mathbb{R}^d$, $m\in \mathcal{P}^2(\mathbb{R}^d)$, $u\in U$,
   $$\|f(t,x,m,u)\| \leq M(1+\|x\|+\varsigma(m)), $$
$$\|b^n(t,x,m,u)\| \leq M(1+\|x\|+\varsigma(m)), $$
$$|\Sigma^n(t,x,m,u)|\leq M(1+\|x\|^2+\varsigma^2(m)),$$
$$|g(t,x,m,u)| \leq M(1+\|x\|+\varkappa(\varsigma(m))), $$ where $\varkappa:[0,+\infty)\rightarrow [0,+\infty)$ is a strictly increasing function;
\item there exists a constant $R$ such that \begin{equation*}\begin{split}|\sigma(x',m')-\sigma(x'',m'')|&\leq R(\|x'-x''\|+W_2(m',m''))\\\cdot (1+&\|x'\|+\|x''\|+\varkappa(\varsigma(m'))+\varkappa(\varsigma(m''))), \end{split}\end{equation*}
\begin{equation*}\begin{split}|g(t,x',m',u)-g(t,x'',m'',u)|&\leq R(\|x'-x''\|+W_2(m',m''))\\\cdot (1+&\|x'\|+\|x''\|+\varkappa(\varsigma(m'))+\varkappa(\varsigma(m''))). \end{split}\end{equation*}
\end{list}
%Without loss of generality we assume that $MT>1$.

We assume that the generators $L^n$ converge to the generator of the deterministic mean field game $L^*$ in the following sense:
\begin{equation}\label{sigma_vanish}
 \sup_{t\in [0,T],x\in\mathbb{R}^d,m\in\mathcal{P}^2(\mathbb{R}^d)}\frac{ \Sigma^n(t,x,m,u)}{1+\|x\|^2+\varsigma^2(m)}\rightarrow 0\mbox{ as }n\rightarrow \infty;
\end{equation}
\begin{equation}\label{evolotion_converge}
      \sup_{t\in [0,T],x\in\mathbb{R}^d,m\in\mathcal{P}^2(\mathbb{R}^d)} \frac{\|b^n(t,x,m,u)-f(t,x,m,u)\|}{(1+\|x\|+\varsigma(m))} \rightarrow 0\mbox{ as }n\rightarrow \infty.
\end{equation}
Additionally, we assume that the initial distributions $m_0^n$ converge to  the probability $m_0$ in the 2-Wasserstein metric i.e.
   \begin{equation}\label{distrib_converge}
   W_2(m_0^n,m_0)\rightarrow 0\mbox{ as }n\rightarrow \infty.
   \end{equation}

The main result of the paper is the following.
\begin{theorem}\label{th_main}
Assume that  conditions (\ref{sigma_vanish})--(\ref{distrib_converge}) hold true and, for each natural~$n$, the  pair $(V^n,\zeta^n)\in C([0,T]\times\rd)\times\mathcal{M}$ solves the system (\ref{HJ_L_n})--(\ref{cond_L_n}) in the probabilistic sense.

Then there exist a pair $(V^*,\mu^*)\in C([0,T]\times\rd)\times\mathcal{M}$ that is a minimax solution to the system (\ref{HJ}), (\ref{kinetic})   and a sequence $\{n_l\}_{l=1}^\infty$ such that
\begin{enumerate}
  \item $\sup_{t\in[0,T]}W_2(\zeta^{n_l}[t],\mu^*[t])\rightarrow 0$ as $l\rightarrow \infty$;
  \item $$\lim_{l\rightarrow\infty}\sup_{(t,x)\in[0,T]\times\rd}\frac{|V^{n_l}(t,x)-V^*(t,x)|}{1+\|x\|^2}=0$$
\end{enumerate}
\end{theorem}
\begin{corollary}\label{coroll_class} If, for each natural number $n$, the pair $(V^n,\zeta^n)$ is a classical solution to system (\ref{HJ_L_n})--(\ref{cond_L_n}), then the conclusion of Theorem \ref{th_main} holds true.
\end{corollary}

\begin{corollary}\label{coroll_link}{}
Any solution in the probabilistic sense is a minimax solution.

 If $(V^*,\mu^*)$ satisfies conditions of  Proposition \ref{pl_minimax_suff} and the set
 $\{(f(t,x,m,u),g(t,x,m,u)):u\in U\}$ is convex for all $t\in [0,T]$, $x\in\mathbb{R}^d$, $m\in\mathcal{P}^2(\mathbb{R}^d)$, then $(V^*,\mu^*)$ is a solution to system (\ref{HJ}), (\ref{kinetic}) in the probabilistic sense.
\end{corollary}

\section{Uniform bounds for flows of probabilities }\label{sect_bounds}

We choose and fix the control system $(\Omega^n,\mathcal{F}^n,\{\mathcal{F}^n_t\}_{t\in [0,T]},P^n,u^n,Y^n)$ that is  admissible for the generator $L^n$ and the flow of probabilities $\zeta^n$. Furthermore, we assume that the tuple $(\Omega^n,\mathcal{F}^n,\{\mathcal{F}^n_t\}_{t\in [0,T]},P^n,u^n,Y^n)$ satisfies condition (iii) of Definition~\ref{def_probability}.  Below $\mathbb{E}^n$ denotes expectation corresponding to the probability $P^n$.

There exists (see \cite{Sznitman}) a $\{\mathcal{F}_t^n\}_{t\in [0,T]}$-adapted c\`{a}dl\`{a}g stochastic process  $X^n$    with values in $\mathbb{R}^d$ and continuous sample paths and a flow  $\mu^n$ such that
\begin{equation}\label{motion_eq}
\frac{d}{dt}X^n(t,\omega)=f(t,X^n(t,\omega),\mu^n[t],u^n(t,\omega)),\  \ X^n(0,\omega)=X_0^n(\omega),
\end{equation}
\begin{equation}\label{mu_eq}
  \mu^n[t]=\mathrm{Law}(X^n(t,\cdot)).
\end{equation} Here $X^n_0(\omega)\triangleq Y^n(0,\omega)$. Note that $\mathrm{Law}(X^n_0)=m_0^n$.
Further, put
\begin{equation}\label{running_payoff_def}
  \mathcal{X}^n(t,\omega)\triangleq \int_0^t g(\tau,X^n(\tau,\omega),\mu^n[\tau],u^n(\tau,\omega))d\tau.
\end{equation} Let $\mathrm{traj}^n:\Omega^n\rightarrow \mathcal{C}$ be defined by the following rule:
$$\mathrm{traj}^n(\omega)\triangleq (X^n(\cdot,\omega),\mathcal{X}^n(\cdot,\omega)). $$ Put
\begin{equation}\label{chi_eq}
  \chi^n\triangleq \mathrm{traj}^n{}_\# P^n.
\end{equation} We have that
\begin{equation}\label{mu_n_chi_n}
\mu^n[t]=e_t{}_\#\chi^n.
\end{equation}

Note that,
$$\varsigma^2(m_0^n)\leq 2W^2(m_0^n,m_0)+2\varsigma^2(m_0). $$ Therefore there exists a constant $M_0$ such that
$$\varsigma^2(m_0^n)\leq M_0. $$

\begin{lemma}\label{lm_zeta_bound}There exists a constant $C_1$ such that, for any $t\in [0,T]$,
$\varsigma^2(\zeta^n[t])\leq C_1$.
\end{lemma}
\begin{proof}
Formulas (\ref{phi_s_t_martingale_L_n}), (\ref{generator_qaud}) and condition (A4) yield that
\begin{equation*}
\begin{split}
  \mathbb{E}\|Y^n(t)\|^2=\mathbb{E}\|Y^n(0)\|^2+&\mathbb{E}\int_{0}^t(\Sigma^n(\tau,Y^n(\tau),\zeta^n[\tau],u^n(\tau))d\tau\\ +2&\mathbb{E}\int_{0}^t\langle b^n(\tau,Y^n(\tau),\zeta^n[\tau],u^n(\tau)),Y^n(\tau)\rangle d\tau\\
  \leq\varsigma^2(m_0^n)+M&\int_{0}^t\mathbb{E}(1+\|Y^n(\tau)\|^2+\varsigma^2(\zeta^n[\tau]))d\tau \\+2M&\int_{0}^t\mathbb{E} \left((1+\|Y^n(\tau)\|+\varsigma(\zeta^n[\tau]))\|Y^n(\tau)\|\right)d\tau\\
\leq   \varsigma^2(m_0^n)+M&\int_{0}^t\mathbb{E}(2+5\|Y^n(\tau)\|^2+2\varsigma^2(\zeta^n[\tau]))d\tau.
\end{split}
\end{equation*}
Since $\mathbb{E}\|Y^n(t)\|^2=\varsigma^2(\zeta^n[t])$ we get the inequality
$$\varsigma^2(\zeta^n[t])\leq \varsigma^2(m_0^n)+2Mt+7M\int_{0}^t\varsigma^2(\zeta^n[\tau])d\tau. $$ Using Gronwall inequality we get the conclusion of the Lemma with $C_1=(M_0+2MT)e^{7MT}$.
\end{proof}

\begin{lemma}\label{lm_mu_estima}There exists a constant $C_2$ such that, for any $t\in [0,T]$,
$\varsigma^2(\mu^n[t])\leq C_2.$
\end{lemma}
The proof of Lemma \ref{lm_mu_estima} is analogous to the proof of Lemma \ref{lm_zeta_bound}.

\begin{lemma}\label{lm_zeta_admiss} There exists a constant $C_3$ such that, if $s\in [0,T]$, $(\Omega,\mathcal{F},\{\mathcal{F}_t\}_{t\in [s,T]},P,u,Y)$ is a control system on $[s,T]$ admissible for the generator $L^n$ and the flow of probabilities $\zeta^n$,  then, for any $t\in [s,T]$,
$$\mathbb{E}\|Y(t)\|^2\leq C_3(1+\mathbb{E}\|Y(s)\|^2). $$
\end{lemma}
\begin{proof}
By (\ref{phi_s_t_martingale_L_n}), (\ref{generator_qaud}) and condition (A4) we have that
\begin{equation*}
\begin{split}
  \mathbb{E}\|Y(t)\|^2=\mathbb{E}\|Y(s)\|^2&+\mathbb{E}\int_{s}^t(\Sigma^n(\tau,Y(\tau),\zeta^n[\tau],u(\tau))d\tau\\ +2\mathbb{E}&\int_{s}^t\langle b^n(\tau,Y(\tau),\zeta^n[\tau],u(\tau)),Y(\tau)\rangle d\tau\\
  \leq\mathbb{E}\|Y(s)\|^2&+M\int_{s}^t\mathbb{E}(1+\|Y(\tau)\|^2+\varsigma^2(\zeta^n[\tau]))d\tau \\+2M&\int_{s}^t\mathbb{E} \left((\|Y(\tau)\|+\|Y(\tau)\|+\varsigma(\zeta^n[t]))\|Y(\tau)\|\right)d\tau\\
\leq   \mathbb{E}\|Y(s)\|^2&+M\int_{s}^t\mathbb{E}(2+5\|Y(\tau)\|^2+2\varsigma^2(\zeta^n[\tau]))d\tau.
\end{split}
\end{equation*} Using Lemma \ref{lm_zeta_bound}  and Gronwall's inequality, we obtain the conclusion of the Lemma with
$C_3\triangleq 2MT(1+C_1)e^{5MT}$.
\end{proof}

\begin{lemma}\label{lm_mu_admiss} There exists a constant $C_4$ such that, if $s\in [0,T]$, $(\Omega,\mathcal{F},\{\mathcal{F}_t\}_{t\in [s,T]},P)$ is a probability space, control $u$ is a $\{\mathcal{F}_t\}_{t\in [s,T]}$-progressively measurable, and $X$ is a stochastic process satisfying
\begin{equation}\label{deq:X_omega}
\frac{d}{dt}X(t,\omega)=f(t,X(t,\omega),\mu^n[t],u(t,\omega)),
\end{equation}
  then, for any $t\in [s,T]$,
$$\mathbb{E}\|X(t)\|^2\leq C_4(1+\mathbb{E}\|X(s)\|^2). $$
\end{lemma}
The proof of this Lemma is analogous to the proof of Lemma \ref{lm_zeta_admiss}.

\section{Distance between flows of probabilities}\label{sect_distance}
Without loss of generality we may assume that there exists a sequence $\{\varepsilon^n\}_{n=1}^\infty$ such that
\begin{itemize}
\item  $\varepsilon^n\leq 1$, $\varepsilon^n\rightarrow 0$ as $n\rightarrow\infty$;
\item
$$\sup_{t\in [0,T],x\in\mathbb{R}^d,m\in\mathcal{P}^2(\mathbb{R}^d)}\frac{ \Sigma^n(t,x,m,u)}{1+\|x\|^2+\varsigma^2(m)}\leq \varepsilon^n;  $$
\item
      $$\sup_{t\in [0,T],x\in\mathbb{R}^d,m\in\mathcal{P}^2(\mathbb{R}^d)} \frac{\|b^n(t,x,m,u)-f(t,x,m,u)\|}{1+\|x\|+\varsigma(m)}\leq\sqrt{\varepsilon^n};$$
\item $W_2^2(m_0^n,m_0)\leq {\varepsilon^n}$.
\end{itemize}

\begin{lemma}\label{lm_closeness}There exists a constant $C_5$ such that, for any $t\in [0,T]$,
$$W_2^2(\zeta^n[t],\mu^n[t])\leq C_5\varepsilon^n. $$
\end{lemma}
\begin{proof}
To simplify the designations let us introduce the following stochastic processes
 $$\widetilde{\Sigma}^n(\tau)=\Sigma^n(\tau,Y^n(\tau),\zeta^n[\tau],u^n(\tau)), $$
 $$%\begin{multline*}
 \tilde{b}^n(\tau)\triangleq b^n(\tau,Y^n(\tau),\zeta^n[\tau],u^n(\tau))$$%\\=f^n(t,Y^n(\theta),\zeta^n[\theta],u^n(\theta))+\int_{\mathbb{R}^d\setminus B_1} y\zeta^n(t,Y^n(\theta),\zeta^n[\theta],u^n(\theta),dy), \end{multline*}
 $$\hat{f}^n(\tau)\triangleq f(\tau,X^n(\tau),\mu^n[\tau],u^n(\tau)). $$

For $\theta,r\in [0,T]$, $\theta\geq r$ we have that
\begin{equation*}%\label{YX_ts_estima}
\begin{split}
  \|Y^n(\theta)-X^n(\theta)\|^2\ \ \ \ \ \ &{}\\
  = \|(Y^n(\theta)-Y^n(r))&-(X^n(\theta)-X^n(r))+(Y^n(r)-X^n(r))\|^2 \\
  =\|Y^n(r)-X^n(r)\|&{}^2+\|Y^n(\theta)-Y^n(r)\|^2+\|X^n(\theta)-X^n(r)\|^2\\+
  2\langle Y^n(r)-X^n(&r),Y^n(\theta)-Y^n(r) \rangle-2\langle Y^n(r)-X^n(r),X^n(\theta)-X^n(r) \rangle\\
  -2\langle Y^n(\theta)-Y^n(&r),X^n(\theta)-X^n(r)\rangle
  \\\leq \|Y^n(r)-X^n(r)\|&{}^2+ 2\|Y^n(\theta)-Y^n(r)\|^2+2\|X^n(\theta)-X^n(r)\|^2\\+
  2\langle Y^n(r)-X^n(&r),Y^n(\theta)-Y^n(r) \rangle-2\langle Y^n(r)-X^n(r),X^n(\theta)-X^n(r) \rangle.
\end{split}
\end{equation*}
Hence,
\begin{equation}\label{YX_ts_expect_estima}
\begin{split}
  \mathbb{E}^n\|Y^n(\theta)-X^n&(\theta)\|^2\\
  \leq \mathbb{E}^n\|Y^n &(r)-X^n(r)\|^2
  +2\mathbb{E}^n \|Y^n(\theta)-Y^n(r)\|^2\\+2\mathbb{E}^n&\|X^n(\theta)-X^n(r)\|^2
  + 2\mathbb{E}^n\langle Y^n(r)-X^n(r),Y^n(\theta)-Y^n(r) \rangle\\-2\mathbb{E}^n&\langle Y^n(r)-X^n(r),X^n(\theta)-X^n(r) \rangle.
\end{split}
  \end{equation}

Now let us estimate $\mathbb{E}^n\|Y^n(\theta)-Y^n(r)\|^2$.  Recall that
 \begin{equation}\label{varsigma_expect_equality}
 \mathbb{E}^n\|Y^n(\theta)\|^2=\varsigma^2(\zeta^n[\theta]),\ \ \mathbb{E}^n\|X^n(\theta)\|^2= \varsigma^2(\mu^n[\theta]).
 \end{equation} 
Further, (\ref{phi_s_t_martingale_L_n}), (\ref{generator_qaud}) and condition (A4) imply that 
\begin{equation*}
\begin{split}
  \mathbb{E}^n\|&Y^n(\theta)-Y^n(r)\|^2\\
   &= \mathbb{E}^n\int_r^\theta \widetilde{\Sigma}^n(\tau) d\tau+2\mathbb{E}^n\int_r^\theta\langle \tilde{b}^n(\tau),Y^n(\tau)-Y^n(r) \rangle d\tau
  \\ &\leq
  \varepsilon^n \left((\theta-r)+2\int_r^\theta\varsigma^2(\zeta^n(\tau))d\tau\right)+2\mathbb{E}^n\int_r^\theta\langle \tilde{b}^n(\tau) ,Y^n(\tau)-Y^n(r) \rangle d\tau.
\end{split}
\end{equation*} Using Lemma \ref{lm_zeta_bound}, we get the inequality
\begin{equation}\label{Y_ts_first_estima}
  \mathbb{E}^n\|Y^n(\theta)-Y^n(r)\|^2
  \leq \varepsilon^nc_1'(\theta-r)+2\mathbb{E}^n\int_r^\theta\langle \tilde{b}^n(\tau) ,Y^n(\tau)-Y^n(r) \rangle d\tau,
\end{equation} where $c_1'=1+2C_1$. It follows from condition (A4) and Lemma \ref{lm_zeta_bound} that
\begin{equation}\label{f_estima}
\mathbb{E}^n\|\tilde{b}^n(\tau)\|^2\leq M^2(2+3\varsigma^2(\zeta^n[\tau])+3\mathbb{E}^n\|Y^n(\tau)\|^2)\leq M^2(2+6\varsigma^2(\zeta^n[\tau])) \leq c_2'.
\end{equation} Here $c_2'\triangleq M^2(2+6C_1)$.
 Therefore
$$\mathbb{E}^n\|Y^n(\theta)-Y^n(r)\|^2
  \leq (\varepsilon^nc_1'+c_2')(\theta-s)+\int_{r}^\theta\mathbb{E}^n\|Y^n(\tau,\cdot)-Y^n(r,\cdot)\|^2d\tau. $$
  Using Gronwall inequality, we obtain the following estimate
  $$\mathbb{E}^n\|Y^n(\theta)-Y^n(r)\|^2\leq c_3'(\theta-r). $$
Here constant $c_3'$ is   equal to $(c_1'+c_2')e^T$.  From this, (\ref{Y_ts_first_estima}) and (\ref{f_estima}) we get that
\begin{equation*}
\begin{split}
  \mathbb{E}^n\|Y^n(\theta)-&Y^n(r)\|^2
  \\\leq \varepsilon^n&c_1'(\theta-r)
  +2\int_r^\theta \sqrt{\mathbb{E}^n\| \tilde{b}^n(\tau)\|^2}\sqrt{\mathbb{E}^n\|Y^n(\tau)-Y^n(r)\|^2} d\tau
  \\ \leq \varepsilon^n&c_1'(\theta-r)+2\int_r^\theta\sqrt{c_2'}\sqrt{c_3'}\sqrt{\tau-r}d\tau.
\end{split}
\end{equation*} Therefore, we conclude that,
 for  $c_4'\triangleq 2\sqrt{c_2'}\sqrt{c_3'}/3$,
\begin{equation}\label{exp_Y_ts_estima}
\mathbb{E}^n\|Y^n(\theta)-Y^n(r)\|^2\leq \varepsilon^n c_1'(\theta-r)+c_4'(\theta-r)^{3/2}.
\end{equation}

Now let us estimate $\mathbb{E}^n\|X^n(\theta)-X^n(r)\|^2$.

By Jensen's inequality and Lemma \ref{lm_mu_estima} we have that $$\mathbb{E}^n\|X^n(\theta)\|\leq\sqrt{\mathbb{E}^n(\|X^n(\theta)\|^2)}=\varsigma(\mu^n[\theta])\leq \sqrt{C_2}. $$ Thus, we obtain that
\begin{equation}\label{exp_X_ts_estima}
\begin{split}
  E^n\|X^n(\theta)-&X^n(r)\|^2=\mathbb{E}^n\left\|\int_r^\theta f(\tau,X^n(\tau),\mu^n[\tau],u^n(\tau))d\tau\right\|^2 \\
  \leq  &M^2\int_r^\theta\mathbb{E}^n\left(1+\|X^n(\tau)\|+\varsigma(\mu^n[\tau])\right)^2d\tau\leq c_5'(\theta-r)^2.
  \end{split}\end{equation}
Here $c_5'\triangleq M^2(3+5C_2).$

Now let us estimate the remaining terms of the right-hand side of (\ref{YX_ts_expect_estima}).

Since the control system $(\Omega^n,\mathcal{F}^n,\{\mathcal{F}_t^n\}_{t\in[0,T]},P^n,u^n,Y^n)$ is admissible for $L^n$ and $\zeta^n$ by (\ref{phi_s_t_martingale_L_n}), (\ref{phi_s_t_martingale_L_star}), (\ref{generator_lin}), (\ref{generator_det_lin}) we have that
\begin{equation*}
\begin{split}
 \mathbb{E}^n\langle Y^n(r)-X^n(r),Y^n(\theta)-Y^n(r) \rangle-\mathbb{E}^n\langle Y^n(r)-X^n(r),X^n(\theta)-&X^n(r) \rangle\\= \mathbb{E}^n\int_r^\theta \langle Y^n(r)-X^n(r), \tilde{b}^n(\tau)-&\hat{f}^n(\tau)\rangle d\tau.
\end{split}
\end{equation*}

Further,
\begin{equation*}
\begin{split}
  \|\tilde{b}^n(\tau)-\hat{f}^n&(\tau)\|
  = \|b^n(\tau,Y^n(\tau),\zeta^n[\tau],u^n(\tau))-f(\tau,X^n(\tau),\mu^n[\tau],u^n(\tau))\|\\
   \leq \|b^n&(\tau,Y^n(\tau),\zeta^n[\tau],u^n(\tau))-f(\tau,Y^n(\tau),\zeta^n[\tau],u^n(\tau))\|\\
   &+\|f(\tau,Y^n(\tau),\zeta^n[\tau],u^n(\tau))-f(r,Y^n(r),\zeta^n[\tau],u^n(\tau))\|\\
   &+\|f(r,Y^n(r),\zeta^n[\tau],u^n(\tau))-f(r,X^n(r),\mu^n[\tau],u^n(\tau))\|\\
   &+   \|f(r,X^n(r),\mu^n[\tau],u^n(\tau))-f(\tau,X^n(\tau),\mu^n[\tau],u^n(\tau))\|
   \\ \leq \sqrt{\varepsilon^n}&(1+\|Y^n(\tau)\|+\varsigma(\zeta^n[\tau]))+2\alpha(\tau-r)\\&+
    K \|Y^n(r)-X^n(r)\|+KW_2(\zeta^n[\tau],\mu^n[\tau])
    \\&+K\|Y^n(\tau)-Y^n(r)\|+K\|X^n(\tau)-X^n(r)\|.
\end{split}
\end{equation*}
Hence, using  Lemma \ref{lm_zeta_bound} and estimates (\ref{exp_Y_ts_estima}), (\ref{exp_X_ts_estima}), we get
\begin{equation}\label{lin_estima_first}
\begin{split}\mathbb{E}^n\langle Y^n(r)-X^n(r),Y^n(\theta)-Y^n(r) \rangle-\mathbb{E}^n\langle Y^n(r)-X^n(r),X^n(\theta)-X^n&(r) \rangle\\
\leq \frac{5K+5}{2}(\theta-r)\|Y^n(r)-X^n(r)\|^2+\frac{K}{2}\int_r^\theta W^2_2(\zeta^n[\tau],&\mu^n[\tau])d\tau\\
+\varepsilon^n c_6'(\theta-r)+\alpha_1(\theta-&r)(\theta-r).
 \end{split}\end{equation} where $c_6'=1/2+C_1$  and $\alpha_1(\delta)=(\alpha(\delta))^2+K\left(c_1'\delta+c_4'\delta^{3/2}+c_5'\delta^2\right)/2. $ We have that $\alpha_1(\delta)\rightarrow 0$ as $\delta\rightarrow 0$.

Combining (\ref{YX_ts_expect_estima}) and (\ref{exp_Y_ts_estima})--(\ref{lin_estima_first}) we get the inequality
\begin{equation}\label{prefinal_estima_YX_ts}
\begin{split}
\mathbb{E}^n\|Y^n(\theta)-X^n(\theta)\|^2\hspace{100pt}&{}\\\leq (1+c_7'(\theta-r))\mathbb{E}^n\|Y^n(r)-X^n(r)&\|^2+ K\int_r^\theta W_2^2(\zeta^n[\tau],\mu^n[\tau])d\tau\\&+\alpha_2(\theta-r)(\theta-r)+\varepsilon^n c_8'(\theta-r)
\end{split}
\end{equation} for \begin{equation}\label{beta_def}
c_7'=5K+5,
\end{equation}  $c_8'=c_1'+c_6'$ and the function $\alpha_2(\delta)\triangleq \alpha_1(\delta)+c_4'\delta^{1/2}+c_5'\delta$. Note that  $\alpha_2(\delta)\rightarrow 0$ as $\delta\rightarrow 0$.

Now let $t\in [0,T]$. For a natural number $N$, set $t^i_N\triangleq it/N$. Applying inequality~(\ref{prefinal_estima_YX_ts}) sequentially for $r=t^i_N$, $\theta=t^{i+1}_N$ we conclude that
\begin{equation*}
\begin{split}
\mathbb{E}\|Y^n(t)-X^n(t)\|^2\leq e^{c_7' t}\mathbb{E}^n\|Y^n(0)-X^n(0)\|^2+  c_9'&\int_0^{t} W_2^2(\zeta^n[\tau],\mu^n[\tau])d\tau\\&+\varepsilon^n c_{10}'t+\alpha_3(T/N)t,
\end{split}
\end{equation*}
where $\alpha_3(\delta)=e^{c_7' T}\alpha_2(\delta)$, $c'_9=e^{c_7' T}K$, $c_{10}'=e^{c_7' T}c_8'$. Since $\mathbb{E}^n\|Y^n(0)-X^n(0)\|^2=0$ we have that
$$%\begin{multline}\label{prefinal_estima_YX_ts}
\mathbb{E}\|Y^n(t)-X^n(t)\|^2\\\leq   c_9'\int_0^{t} W_2^2(\zeta^n[\tau],\mu^n[\tau])d\tau+\varepsilon^n c_{10}'t+\alpha_3(T/N)t.
$$ Passing to the limit as $N\rightarrow \infty$ we get the inequality
\begin{equation*}
\mathbb{E}\|Y^n(t)-X^n(t)\|^2\\\leq  c_9'\int_0^{t} W_2^2(\zeta^n[\tau],\mu^n[\tau])d\tau+\varepsilon^n c_{10}'t.
\end{equation*}
Since
$$W_2^2(\zeta^n[t],\mu^n[t])\leq \mathbb{E}\|Y^n(t)-X^n(t)\|^2,$$
we have that
$$W_2^2(\zeta^n[t],\mu^n[t])\leq  c_9'\int_0^{t} W_2^2(\zeta^n[\tau],\mu^n[\tau])d\tau+ \varepsilon^n c'_{10}t.$$
Using Gronwall inequality, we obtain the conclusion of the Lemma with $C_5=c_{10}'Te^{c'_9 T}$.

\end{proof}

The following Lemma is  concerned with the distance between admissible controlled process and the control system governed by the differential equation depending on the flow of probabilities $\mu^n$.

\begin{lemma}\label{lm_motions} There exists a constant $C_6$ such that, for any $s\in[0,T]$, $\xi\in\mathbb{R}^d$,
if
\begin{itemize}
\item the control system
$(\Omega,\mathcal{F},\mathcal{F}_{t\in[s,T]},{P},u,Y)$ is  admissible for $L^n$ and $\zeta^n$,
\item $Y(s)=\xi$ $P$-a.s.,
\item the process $X$ satisfies the condition
\begin{equation*}
\frac{d}{dt}X(t,\omega)=f(t,X(t,\omega),\mu^n[t],u(t,\omega)),\ \ X(s,\omega)=\xi,
\end{equation*}
\end{itemize}
then
$$\mathbb{E}\|Y(t)-X(t)\|^2\leq C_6\varepsilon_n(1+\|\xi\|^2). $$
Here $\mathbb{E}$ denotes the  expectation corresponding to probability $P$.
\end{lemma}
\begin{proof}
By Lemmas \ref{lm_zeta_admiss} and \ref{lm_mu_admiss} we get that there exist constants $c_1''$ and $c_2''$ such that, for any $\theta,r\in [0,T]$,
$$
\mathbb{E}\|Y(\theta)-Y(r)\|^2\leq c_1''[\varepsilon^n(1+\|\xi\|^2)+(\theta-r)^{1/2}](\theta-r),$$ $$\mathbb{E}\|X(\theta)-X(r)\|^2\leq c_2''(1+\|\xi\|^2)(\theta-r)^2.$$
 These inequalities can be deduced in the same way as inequalities (\ref{exp_Y_ts_estima}), (\ref{exp_X_ts_estima}).

For  $\theta,r\in [0,T]$, $\theta\geq r$, we have that
\begin{equation*}
\begin{split}
  \mathbb{E}\|Y(\theta)-&X(\theta)\|^2 \\
  \leq \mathbb{E}\|&Y(r)-X(r)\|^2+2\mathbb{E}\|Y(\theta)-Y(r)\|^2+2\mathbb{E}\|X(\theta)-X(r)\|^2\\
  &+  2\mathbb{E} \left\langle Y(r)-X(r),Y(\theta)-Y(r)\right\rangle -2\mathbb{E} \left\langle Y(r)-X(r),X(\theta)-x(r)\right\rangle.
\end{split}
\end{equation*}
Estimating the right-hand side as in the proof of Lemma \ref{lm_closeness}, using (\ref{phi_s_t_martingale_L_n}), (\ref{phi_s_t_martingale_L_star}), (\ref{generator_lin}), (\ref{generator_det_lin}) and Lemmas \ref{lm_zeta_bound}--\ref{lm_mu_admiss} we get that
\begin{equation*}
\begin{split}
  \mathbb{E}\|Y(\theta)-X(\theta)\|^2\hspace{100pt}&{} \\
  \leq (1+c_7'(\theta-r))\mathbb{E}^n\|Y(r)-X(r)&\|^2+ K\int_{s}^{t} W_2^2(\zeta^n[\tau],\mu^n[\tau])d\tau\\&+\varepsilon^n c_3''(1+\|\xi\|^2)(\theta-r)+\alpha_4(\theta-r,\xi).
\end{split}
\end{equation*} Here $c_7'$ is given by (\ref{beta_def}), $c_3''$ is a constant, $\alpha_4(\cdot,\cdot)$ is a function such that $\alpha_4(\delta,\xi)\rightarrow 0$ as $\delta\rightarrow 0$ for any $\xi\in\mathbb{R}^d$. Hence, by Lemma \ref{lm_closeness} we obtain the following estimate
\begin{equation}
\begin{split}\label{YX_diff_estima_t_prime}\mathbb{E}\|Y(\theta)-&X(\theta)\|^2
  \leq (1+c_7' (\theta-r))\mathbb{E}\|Y(r)-X(r)\|^2\\&+\alpha_4(\theta-r,\xi)(\theta-r) +\varepsilon^n (c_3''(1+\|\xi\|^2)+KC_5)(\theta-r). \end{split}\end{equation}
Now let $t\in[0,T]$, $t_N^i\triangleq i t/N$. Using the inequality (\ref{YX_diff_estima_t_prime}) sequentially, we get the inequality
\begin{equation*}\begin{split} \mathbb{E}\|&Y(t_N^i)-X(t_N^i)\|^2
  \leq e^{c_7' (t_N^i-t_N^0)}\mathbb{E}\|Y(t_N^0)-X(t_N^0)\|^2\\&+e^{c_7' (t_N^i-t_N^0)}\alpha_4(T/N,\xi)(t_N^i-t_N^0) +\varepsilon^n (c_3''(1+\|\xi\|^2)+KC_5)e^{c_7'(t_N^i-t_N^0)}(t_N^i-t_N^0).\end{split}\end{equation*}
Hence, $$ \mathbb{E}\|Y(t)-X(t)\|^2
  \leq e^{c_7' T}\alpha_4(T/N,\xi)T +\varepsilon^n (c_3''(1+\|\xi\|^2)+KC_5)e^{c_7' T}T.$$ Letting $N\rightarrow \infty$ we get the conclusion of the Lemma.
\end{proof}

\section{Limit flow of probabilities}\label{sect_limit}
This section is devoted to the proofs of Theorem \ref{th_main} and Corollaries \ref{coroll_class}, \ref{coroll_link}.

\begin{proof}[Proof of Theorem \ref{th_main}]

First, let us show that the family of probabilities $\{\chi^n\}_{n=1}^\infty$ is tight. Indeed, by (\ref{chi_eq}), (\ref{mu_n_chi_n})  $$\chi^n={\rm traj}^n{}_{\#} P^n, \ \ \mu^n[t]={e_t}_{\#}\chi^n=X(t,\cdot)_\# P^n.$$
This means that, for any continuous function $\varphi:\mathcal{C}\rightarrow \mathbb{R}$,
$$\int_{\mathcal{C}}\varphi(x(\cdot),z(\cdot))\chi^n(d(x(\cdot),z(\cdot)))= \mathbb{E}^n\varphi(X^n(\cdot),\mathcal{X}^n(\cdot)). $$
We have that
$$\mathbb{E}^n(\|X^n(0)\|^2+|\mathcal{X}^n(0)|^2)=\mathbb{E}^n\|X^n(0)\|^2\leq \varsigma^2(m_0^n). $$
 Further,
  $$\|X^n(t'')-X^n(t')\|\leq M\left(1+\sup_{t}\|X^n(t)\|+\sup_{t}\varsigma(\mu^n(t))\right)|t''-t'|. $$
  In addition,
  $$|\mathcal{X}^n(t'')-\mathcal{X}^n(t')|\leq M\left(1+\sup_{t}\|X^n(t)\|+\sup_{t}\varkappa(\varsigma(\mu^n(t)))\right)(t''-t'). $$
Since $\mathbb{E}\|X^n(t)\|\leq \sqrt{\mathbb{E}\|X^n(t)\|^2}=\varsigma(\mu^n[t])$ by Lemma \ref{lm_mu_estima} there exists a constant $c_1^*$ such that
  $$\mathbb{E}^n(\|X^n(t'')-X^n(t')\|+|\mathcal{X}^n(t'')-\mathcal{X}^n(t')|)\leq c_1^*|t''-t'|.$$

Therefore by \cite[Theorem 7.3]{Billingsley} the sequence $\{\chi^n\}_{n=1}^\infty$ is tight.

Now let us prove that $\{\chi^n\}$ have uniformly integrable second moments.  Condition (A\ref{cond_growth}) and Lemma \ref{lm_mu_estima} imply that there exists a constant $c_1^*$ such that, for any $\omega\in\Omega^n$,
$$\|X^n(\cdot,\omega)\|^2\leq c_1^*(1+\|X^n_0(\omega)\|^2). $$ Further, we have that
$$\|\mathcal{X}^n(\cdot,\omega)\|^2\leq c_2^*(1+\|X^n_0(\omega)\|^2). $$ Thus,
\begin{multline}\label{moment_bound}\int_{\|w(\cdot)\|^2\geq k}\|w(\cdot)\|^2\chi^n(d(w(\cdot)))\\=\int_{\Omega^n:\|X^n(\cdot,\omega)\|^2+\|\mathcal{X}^n(\cdot,\omega)\|^2\geq k}[\|X^n(\cdot,\omega)\|^2+\|\mathcal{X}^n(\cdot,\omega)\|^2]P^n(d\omega)\\\leq
\int_{\Omega^n:\|X^n_0(\omega)\|^2\geq k/(c_1^*+c_2^*)-1}(c_1^*+c_2^*)(1+\|X^n_0(\omega)\|^2)P^n(d\omega)\\=
\int_{\|\xi\|\geq k/(c_1^*+c_2^*)-1}(c_1^*+c_2^*)(1+\|\xi\|^2)m^n_0(d\xi).
 \end{multline} Since    $\{m_0^n\}$ converges to $m_0$ in 2-Wasserstein metric, we have that the sequence $\{m_0^n\}$ has uniformly integrable second moments. This and (\ref{moment_bound}) implies that $\{\chi^n\}$ have uniformly integrable second moments.

Using this  and tightness of $\{\chi^n\}$, we have that that there exist a sequence $n_l$ and a probability $\chi^*\in\mathcal{P}^2(\mathcal{C})$ such that
 \begin{equation}\label{measure_convergence}
 W_2(\chi^{n_l},\chi^*)\rightarrow 0\mbox{ as }l\rightarrow \infty.
\end{equation}
Denote  $\mu^*[t]\triangleq{e_t}_\# \chi^*$.
Further, let $V^*(s,\xi)$ be a value function of problem (\ref{mfg_def:relaxed_payoff}), (\ref{mfg_def:relax_set}) for $\mu=\mu^*$. Below we show that
\begin{list}{(\roman{tmp})}{\usecounter{tmp}}
\item  $\mu^{n_l}$, $V^{n_l}$ converge to $\mu^*$, $V^*$, respectively (see statements 1, 2 of the Theorem);
\item $(V^*,\mu^*)$ is a solution to mean field game (\ref{HJ}), (\ref{kinetic}).
\end{list}

 Inequality (\ref{wass_estima}) yields that
\begin{equation*}
 \sup_{t\in [0,T]}W_2(\mu^*[t],\mu^{n_l}[t])\rightarrow 0\mbox{ as }l\rightarrow \infty.
 \end{equation*}
This and Lemma \ref{lm_closeness} imply that 
\begin{equation}\label{mu_star_limit_zeta}
 \sup_{t\in [0,T]}W_2(\mu^*[t],\zeta^{n_l}[t])\rightarrow 0\mbox{ as }l\rightarrow \infty.
 \end{equation} Thus, statement 1 of the Theorem is proved.

To prove the second statement of the Theorem we introduce the auxiliary function
$\widehat{V}^n(s,\xi)$ that is equal to the value function of problem (\ref{mfg_def:relaxed_payoff}), (\ref{mfg_def:relax_set}) for $\mu=\mu^n$. Now let us estimate $|V^*(s,\xi)-\widehat{V}^n(s,\xi)|$. 

We have that
\begin{equation}\label{V_n_def}
\widehat{V}^n(s,\xi)=\sup\{\sigma(x[T,\mu^n,s,\xi,v],\mu^n[T])+z[T,\mu^n,s,\xi,v]:v\in\mathcal{U}\}, \end{equation}
\begin{equation}\label{V_star_def}
V^*(s,\xi)=\sup\{\sigma(x[T,\mu^*,s,\xi,v],\mu^*[T])+z[T,\mu^*,s,\xi,v]:v\in\mathcal{U}\}.
\end{equation}
Using Gronwall's inequality, condition (A3), and inequality (\ref{wass_estima}), we obtain, for every $v\in\mathcal{U}$, the following estimate:
\begin{equation}\label{estima:x_n_x_star}
\|x[t,\mu^n,s,\xi,v]-x[t,\mu^*,s,\xi,v]\|\leq KTW_2(\chi_n,\chi_*)e^{KT}.
\end{equation}
This, condition (A5), formula (\ref{z_def}), inequality (\ref{x_control_estima}) and Lemma \ref{lm_mu_estima} imply that there exists a constant $c_3^*$ such that
\begin{multline*}
  |(\sigma(x[T,\mu^n,s,\xi,v],\mu^n[T])+z[T,\mu^n,s,\xi,v])\\-(\sigma(x[T,\mu^*,s,\xi,v],\mu^*[T])+ z[T,\mu^*,s,\xi,v])|\\
  \leq c_3^*(1+\|\xi\|)W_2(\chi^n,\chi^*).
\end{multline*} This inequality and representations (\ref{V_n_def}), (\ref{V_star_def}) yield the following  estimate:
\begin{equation}\label{estima:V_n_V_star}
  |\widehat{V}^n(s,\xi)-V^*(s,\xi)|\leq c_3^*(1+\|\xi\|)W_2(\chi^n,\chi^*).
\end{equation}

Further, let us estimate $|\widehat{V}^n(s,\xi)-V^n(s,\xi)|$. To this end we consider a control system $(\Omega,\mathcal{F},\{\mathcal{F}_t\}_{t\in [s,T]},P,u,Y)$ admissible for $L^n$ and $\zeta^n$ such that $Y(s)=\xi$ P-a.s. Let, for each $\omega\in\Omega$, $X(\cdot,\omega)$ satisfy (\ref{deq:X_omega})  and initial condition $X(s,\omega)=\xi$.

Condition (A5) implies that
\begin{multline*}
|\sigma(X(T),\mu^n(T))- \sigma(Y(T),\zeta^n(T))|\\\leq
R(|X(T)-Y(T)|+W_2(\mu^n(T),\zeta^n(T)))\\
\cdot(1+\|X(T)\|+\|Y(T)\|+\varkappa(\varsigma(\mu^n[T]))+\varkappa(\varsigma(\zeta^n[T]))).
\end{multline*}
\begin{multline*}
  |g(t,X(t),\mu^n(t),u(t))-g(t,Y(t),\zeta^n(t),u(t))|\\\leq
  R(|X(t)-Y(t)|+W_2(\mu^n(t),\zeta^n(t)))\\ \cdot(1+\|X(t)\|+\|Y(t)\|+\varkappa(\varsigma(\mu^n[t]))+\varkappa(\varsigma(\zeta^n[t]))).
\end{multline*}
These inequalities and
Lemmas \ref{lm_zeta_admiss}--\ref{lm_motions} yield that, for some constant $c_4^*$, the following inequality holds true:
\begin{equation}\label{ineq:sigma_int_g}
\begin{split}
\mathbb{E}\Bigl|&\sigma(X(T),\mu^n(T))+\int_s^Tg(t,X(t),\mu^n(t),u(t))dt\\&- \sigma(Y(T),\zeta^n(T))-\int_s^Tg(t,Y(t),\zeta^n(t),u(t))dt\Bigr| \leq
c_4^*\sqrt{\varepsilon^n}(1+\|\xi\|^2).
\end{split}
\end{equation}

Further, by (\ref{motion_eq}) we have that
\begin{equation}\label{ineq:hat_V_n_expect}
\widehat{V}^n(s,\xi)\geq \mathbb{E}\left[\sigma(X(T),\mu^n(T))+\int_s^Tg(t,X(t),\mu^n(t),u(t))dt\right].
\end{equation} Recall that $V^n(s,\xi)$ is a value function of the optimization problem for (\ref{objective_function}) over the set of control systems on $[s,T]$ $(\Omega,\mathcal{F},\{\mathcal{F}_t\}_{t\in [s,T]},P,u,Y)$ admissible for $L^n$ and $\zeta^n$ such that $Y(s)=\xi$ P-a.s. Choose $(\Omega^{n,\delta}_{s,\xi},\mathcal{F}^{n,\delta}_{s,\xi}, \{\mathcal{F}^{n,\delta}_{s,\xi,t}\}_{t\in [0,T]},P^{n,\delta}_{s,\xi},u^{n,\delta}_{s,\xi},Y^{n,\delta}_{s,\xi})$ to be $\delta$-optimal for this problem. Further, let $X^{n,\delta}_{s,\xi}$ satisfy (\ref{deq:X_omega})  and initial condition $X(s,\omega)=\xi$. Denoting by $\mathbb{E}^{n,\delta}_{s,\xi}$ the expectation according to $P^{n,\delta}_{s,\xi}$, from (\ref{ineq:sigma_int_g}), (\ref{ineq:hat_V_n_expect}) we obtain that
\begin{equation*}
\begin{split}
\widehat{V}^n&(s,\xi)\\
&\geq \mathbb{E}^{n,\delta}_{s,\xi}\Bigl[\sigma(X^{n,\delta}_{s,\xi}(T),\mu^n(T))+ \int_s^Tg(t,X^{n,\delta}_{s,\xi}(t),\mu^n(t),u^{n,\delta}_{s,\xi}(t))dt\Bigr]\\
&\geq \mathbb{E}^{n,\delta}_{s,\xi}\Bigl[\sigma(Y^{n,\delta}_{s,\xi}(T),\zeta^n(T))+ \int_s^Tg(t,Y^{n,\delta}_{s,\xi}(t),\zeta^n(t),u^{n,\delta}_{s,\xi}(t))dt\Bigr]-c_4^*\sqrt{\varepsilon^n}(1+\|\xi\|^2)\\ &\geq V^n(s,\xi)-c_4^*\sqrt{\varepsilon^n}(1+\|\xi\|^2)-\delta.
\end{split}
\end{equation*}
Letting $\delta\rightarrow 0$,  we get the following:
\begin{equation}\label{ineq:V_n_V_hat_upper}
  \widehat{V}^n(s,\xi)\geq V^n(s,\xi)-c_4^*\sqrt{\varepsilon^n}(1+\|\xi\|^2).
\end{equation}

Now we turn to the opposite estimate. Let $v\in\mathcal{U}$ be a deterministic control. By condition (A0) there exists a filtered probability space $(\Omega,\mathcal{F},\mathcal{F}_{t\in [0,T]},P)$ and a stochastic process $Y$    such that $Y(s)=\xi$ $P$-a.s., and the control system $(\Omega,\mathcal{F},\mathcal{F}_{t\in [0,T]},P,v,Y)$ is admissible for $L^n$, $\zeta^n$. Note that 
\begin{equation}\label{ineq:E_v_V}
\mathbb{E}\left[\sigma(Y(T),\zeta^n(T))+\int_s^Tg(t,Y(t),\zeta^n(t),v(t))dt\right]\leq V^n(s,\xi). \end{equation}

In this case, if $X(\cdot,\omega)$ satisfies (\ref{deq:X_omega}) for $u=v$, and $X(s,\omega)=\xi$, then $X(t,\omega)=x[t,\mu^n,s,\xi,v]$, $$\int_s^Tg(t,X(t),\mu^n(t),v(t))dt=z[t,\mu^n,s,\xi,v].$$ Using this, inequalities (\ref{V_n_def}), (\ref{ineq:E_v_V}) and estimate (\ref{ineq:sigma_int_g}) for $u=v$, we get
\begin{equation}\label{ineq:V_n_V_hat_lower}
  \widehat{V}^n(s,\xi)\leq V^n(s,\xi)+c_4^*\sqrt{\varepsilon^n}(1+\|\xi\|^2).
\end{equation}

Combining this and (\ref{ineq:V_n_V_hat_upper}) we conclude that
\begin{equation*}
  |\widehat{V}^n(s,\xi)-V^n(s,\xi)|\leq c_4^*\sqrt{\varepsilon^n}(1+\|\xi\|^2).
\end{equation*}

This estimate, (\ref{estima:V_n_V_star}) and convergence of $W_2(\chi^{n_l},\chi^*)$ to zero imply the second statement of the Theorem.

Now we  prove that $(V^*,\mu^*)$ is a minimax solution to (\ref{HJ}), (\ref{kinetic}).
First, we show that the support of $\chi^*$  lies in the set ${\rm Sol}_0(\mu^*)$. Recall that ${\rm Sol}_0(\mu^*)$ is the set of functions $(x(\cdot),z(\cdot)):[0,T]\rightarrow \mathbb{R}^d$ satisfying
\begin{equation}\label{limit_incl}
(\dot{x}(t),\dot{z}(t))\in {\rm co}\{(f(t,x(t),\mu^*[t],u),g(t,x,\mu^*[t],u)):u\in U\}.
\end{equation}

Let $(x^*(\cdot),z^*(\cdot))\in {\rm supp}(\chi^*)$. Since the sequence $\{\chi^{n_l}\}_{l=1}^\infty$  converges to $\chi^*$, by \cite[Proposition 5.1.8]{ambrosio_gradient_flows}  there exists a sequence of motions $(x^{n_l}(\cdot),z^{n_l}(\cdot))\in {\rm supp}(\chi^{n_l})$ such that $(x^{n_l}(\cdot),z^{n_l}(\cdot))\rightarrow (x^*(\cdot),z^*(\cdot))$ as $l\rightarrow \infty$. We have that $(x^{n_l}(\cdot),z^{n_l}(\cdot))$ satisfies the differential inclusion
$$(\dot{x}(t),\dot{z}(t))\in {\rm co}\{(f(t,x(t),\mu^{n_l}[t],u),g(t,x(t),\mu^{n_l}[t],u)):u\in U\}. $$ From this and (\ref{mu_star_limit_zeta}) it follows that the pair $(x^*(\cdot),z^*(\cdot))$ satisfies inclusion (\ref{limit_incl}). Consequently, $${\rm supp}(\chi^*)\subset{\rm Sol}_0(\mu^*).$$

By construction we have that $V^*$ is a minimax solution to (\ref{HJ}) and $\mu^*=e_t{}_\#\chi^*$.

It remains to prove (\ref{optimal_trajectories}).
To this end let us introduce the functions $\mathcal{R}^n(s)$ and $\mathcal{R}^*(s)$ by the following rules:
$$\mathcal{R}^n(s)=\int_{\mathcal{C}}[\sigma(x(T),\mu^n[T])+z(T)-z(s)- \widehat{V}^n(s,x(s)) ]\chi^n(d(x(\cdot),z(\cdot))),$$
$$\mathcal{R}^*(s)=\int_{\mathcal{C}}[\sigma(x(T),\mu^*[T])+z(T)-z(s)- V^*(s,x(s)) ]\chi^*(d(x(\cdot),z(\cdot))).$$
We have that, for $(x(\cdot),z(\cdot))\in\mathrm{supp}(\chi^*)\subset\mathrm{Sol}_0(\mu_*)$,
\begin{equation}\label{sigma_Val_estima}
  \sigma(x(T),\mu^*[T])+z(T)-z(s)- V^*(s,x(s))\leq 0.
\end{equation}
Consequently,
\begin{equation}\label{R_estima_upper}
  \mathcal{R}^*(s)\leq 0.
\end{equation}

Now we prove the opposite inequality. First, note that there exists a constant $c_5^*$ such that
\begin{equation*}
\begin{split}
\bigl|\sigma(x&[T,\mu^n,s,\xi_1,v],\mu^n[T])+z[T,\mu^n,s,\xi_1,v]\\&-\sigma(x[T,\mu^n,s,\xi_2,v],\mu^n[T])-z[T,\mu^n,s,\xi_2,v]\bigr|\leq c_5^*\|\xi_1-\xi_2\|(1+\|\xi_1\|+\|\xi_2\|).
\end{split}
\end{equation*}
This   
 implies that
\begin{equation*}\label{ineq:V_hat_lip}
\bigl|\widehat{V}^n(s,\xi_1)-\widehat{V}^n(s,\xi)\bigr|\leq c_5^*\|\xi_1-\xi_2\|(1+\|\xi_1\|+\|\xi_2\|).
\end{equation*}
Therefore, using Lemmas \ref{lm_mu_estima}, \ref{lm_zeta_admiss}, \ref{lm_mu_admiss}, \ref{lm_motions} for
$\Omega=\Omega^n$, $\mathcal{F}=\mathcal{F}^n$, $\mathcal{F}_t=\mathcal{F}^n_t$, $P=P^n$, $u=u^n$,  $Y=Y^n$, and $X=X^n$ we get
\begin{equation}\label{ineq:exp_V_hat_lip}
\mathbb{E}^n\bigl|\widehat{V}^n(s,X^n(s))-\widehat{V}^n(s,Y^n(s))\bigr|\leq c_6^*\sqrt{\varepsilon^n},
\end{equation} where $c_6^*\triangleq c_5^*C_6(\sqrt{1+M_0})(1+\sqrt{C_1}+\sqrt{C_2})$.

The construction of probabilities $\chi^n$ yields that
$$\mathcal{R}^n(s)=\mathbb{E}^n\left[\sigma(X^n(T),\mu^n[T])+\int_s^Tg(t,X^n(t),\mu^n[t],u^n(t))dt- \widehat{V}^n(s,X^n(s))\right]. $$ By  (\ref{ineq:V_n_V_hat_lower})  and (\ref{ineq:exp_V_hat_lip})  we have that
\begin{multline*}\mathcal{R}^n(s)\geq \mathbb{E}^n\left[\sigma(X^n(T),\mu^n[T])+\int_s^Tg(t,X^n(t),\mu^n[t],u^n(t))dt-V^n(s,Y^n(s))\right]\\- (c_6^*+c_7^*)\sqrt{\varepsilon^n}. \end{multline*}
Here (see (\ref{ineq:V_n_V_hat_lower}) and Lemma \ref{lm_mu_estima}) $c_7^*\triangleq c_4^*(1+C_2)$. Using Lemma \ref{lm_mu_estima} and inequality (\ref{ineq:sigma_int_g}) for $\Omega=\Omega^n$, $\mathcal{F}=\mathcal{F}^n$, $\mathcal{F}_t=\mathcal{F}^n_t$, $P=P^n$, $u=u^n$,  $Y=Y^n$, and $X=X^n$, we obtain
\begin{multline*}\mathcal{R}^n(s)\geq \mathbb{E}^n\left[\sigma(Y^n(T),\zeta^n[T])+\int_s^Tg(t,Y^n(t),\zeta^n[t],u^n(t))dt-V^n(s,Y^n(s))\right]\\- (c_6^*+2c_7^*)\sqrt{\varepsilon^n}. \end{multline*}
Finally, condition (iii) of Definition \ref{def_probability} and the choice of $(\Omega^n,\mathcal{F}^n,\{\mathcal{F}^n_t\}_{t\in [0,T]},P^n,u^n,Y^n)$
yield
\begin{equation}\label{R_n_estima}
  \mathcal{R}^n(s)\geq - (c_6^*+2c_7^*)\sqrt{\varepsilon^n}.
\end{equation}

From (\ref{measure_convergence}), (\ref{estima:V_n_V_star}) and Lemma \ref{lm_mu_estima} it follows that
$$\mathcal{R}^{n_l}\rightarrow \mathcal{R}^*. $$ Hence, using (\ref{R_n_estima}) we get that
$\mathcal{R}^*(s)\geq 0$.  This and (\ref{R_estima_upper}) imply the equality ${\mathcal{R}^*(s)=0}$.
Therefore, by inequality (\ref{sigma_Val_estima}) we get that, given $s\in [0,T]$,
$|\sigma(x(T),\mu^*[T])+z(T)-z(s)-V^*(s,x(s)) |=0$ for $\chi^*$-a.e. $(x(\cdot),z(\cdot))$. 

Using Fubini's theorem, we conclude that
$$\int_{\mathcal{C}}\int_0^T\left|\sigma(x(T),\mu^*[T])+z(T)-z(s)- V^*(s,x(s)) \right|ds\chi^*(d(x(\cdot),z(\cdot)))=0. $$

This implies that inclusion (\ref{optimal_trajectories}) is fulfilled. Thus, the pair $(V^*,\mu^*)$ is a minimax solution to (\ref{HJ}), (\ref{kinetic}).

\end{proof}

\begin{proof}[Proof of Corollary 1]
The conclusion of the Corollary directly follows  from Theorem~1 and Proposition 1.
\end{proof}

\begin{proof}[Proof of Corollary 2]
Note that if the pair $(V^*,\mu^*)$ solves (\ref{HJ}), (\ref{kinetic}) in the probabilistic sense, then by Theorem 1 it is a solution to (\ref{HJ}), (\ref{kinetic}) in the minimax sense.

Now assume that $(V^*,\mu^*)$ satisfies conditions of Proposition~\ref{pl_minimax_suff}
and the set
$\{(f(t,x,m,u),g(t,x,m,u)):u\in U\}$ is convex. We have that if $(x(\cdot),z(\cdot))\in {\rm Sol}_0(\mu^*)$, then there exists a deterministic control $u_{x(\cdot),z(\cdot)}$ such that
\begin{equation*}
\dot{x}(t)=f(t,x(t),\mu^*[t],u_{x(\cdot),z(\cdot)}(t)),\ \  \dot{z}(t)=g(t,x(t),\mu^*[t],u_{x(\cdot),z(\cdot)}(t))
\end{equation*}
Put $\Omega\triangleq \mathcal{C}$,  $\mathcal{F}_t\triangleq \mathcal{B}(C([0,t],\mathbb{R}^{d+1}))$ (here $\mathcal{B}$ stands for the Borel $\sigma$-algebra), $\mathcal{F}\triangleq \mathcal{F}_T$, $P\triangleq \chi$. For $\omega=(x(\cdot),z(\cdot))$, set $u(t,\omega)\triangleq u_{x(\cdot),z(\cdot)}$, $Y(t,\omega)\triangleq x(t)$. Thus, $(V^*,\mu^*)$ solves (\ref{HJ}), (\ref{kinetic}) in the probabilistic sense.
\end{proof}

\section*{}
\noindent\textbf{Acknowledgement.}
The author also would like to thank the anonymous referee for her/his valuable comments and suggestions.

\nocite{*}

\bibliography{deterministic_limit}

\end{document}